\documentclass[a4paper,12pt]{article}
\usepackage[utf8]{inputenc}
\usepackage{amssymb}
\usepackage{amsmath}
\usepackage{amsthm}
\usepackage{enumitem}
\usepackage{hyperref}
\hypersetup{
	colorlinks=true,
	linkcolor=black,
	linkbordercolor={1 1 1},
	citecolor=black
}

\newtheorem{thm}{Theorem}
\newtheorem{theorem}[thm]{Theorem}
\newtheorem*{theorem*}{Theorem}
\newtheorem{lem}[thm]{Lemma}
\newtheorem{lemma}[thm]{Lemma}
\newtheorem{cor}[thm]{Corollary}
\newtheorem{prop}[thm]{Proposition}

\newtheorem{claim}[thm]{Claim}

\theoremstyle{definition}

\newtheorem{remark}[thm]{Remark}
\newtheorem{rem}[thm]{Remark}

\numberwithin{equation}{section}
\numberwithin{thm}{section}

\newcommand{\N}{\mathbb{N}}
\newcommand{\R}{\mathbb{R}}
\newcommand{\eps}{\varepsilon}

\newcommand{\IN}{\mathbb{N}}
\newcommand{\IP}{\mathbb{P}}

\newcommand{\IR}{\mathbb{R}}

\newcommand{\cB}{\mathcal{B}}
\newcommand{\cC}{\mathcal{C}}

\newcommand{\cL}{\mathcal{L}}
\newcommand{\cM}{\mathcal{M}}
\newcommand{\cN}{\mathcal{N}}
\newcommand{\cR}{\mathcal{R}}

\newcommand{\cX}{\mathcal{X}}

\DeclareMathOperator{\vdim}{dim_{\rm vc}}
\DeclareMathOperator{\vol}{vol}
\DeclareMathOperator{\dens}{dens}
\DeclareMathOperator{\conv}{conv}
\DeclareMathOperator{\disp}{disp}
\newcommand{\dd}{\,\mathrm{d}}
\newcommand{\Sph}{\mathbb{S}^{d}}
\newcommand{\Sp}{\mathbb{S}}

\title{Minimal dispersion on the sphere}
\author{Alexander E. Litvak\footnote{Department of Mathematical and Statistical Sciences, University of Alberta, Edmonton, AB, Canada, T6G 2G1; {\tt alitvak@ualberta.ca; tszczepa@ualberta.ca}}, Mathias Sonnleitner$^{*,}$\footnote{Current address: Institute of Financial Mathematics and Applied Number Theory, Johannes Kepler University Linz, Altenbergerstrasse 69, 4040 Linz, Austria; {\tt mathias.sonnleitner@jku.at}} and Tomasz Szczepanski$^*$ }

\date{}

\begin{document}

\maketitle

\abstract{
The minimal spherical cap dispersion $\disp_{\cC}(n,d)$ is the largest number $\varepsilon\in (0,1]$ such that,
for every $n$ points  on the $d$-dimensional Euclidean unit sphere $\Sp^d$, there exists a spherical cap with normalized area
$\varepsilon$ not containing any of these points. We study the behavior of $\disp_{\cC}(n,d)$ as $n$ and $d$ grow to infinity.
We develop connections to the problems of sphere covering and approximation of the Euclidean unit ball by inscribed polytopes.
Existing and new results are presented in a unified way. Upper bounds on $\disp_{\cC}(n,d)$ result from choosing the points
independently and uniformly at random and possibly adding some well-separated points to close large gaps. Moreover, we study
dispersion with respect to intersections of caps.

}

\section{Introduction and results}

Given a point set $P=\{x_1,\dots,x_n\}$ on the $d$-dimensional Euclidean unit sphere $\Sp^d\subset \IR^{d+1}$, where $d\ge 1$, define its spherical cap dispersion by
\[
\disp_{\cC}(P)
:=\sup_{C\in \cC}\{\sigma(C)\colon C\cap \{x_1,\dots,x_n\}=\emptyset\},
\]
where $\cC$ denotes the family of spherical caps of $\Sp^d$ and $\sigma$ is the normalized Lebesgue measure on $\Sp^d$, which
we also call the spherical measure.  We define the minimal spherical cap dispersion by
\begin{equation} \label{eq:dispersion}
\disp_{\cC}(n,d)
:=\inf_{x_1,\dots,x_n\in \Sph} \disp_{\cC}(\{x_1,\dots,x_n\}).
\end{equation}
We are interested in the asymptotic behavior of $\disp_{\cC}(n,d)$ as $n$ and $d$ grow to infinity.

In the case of the unit cube $[0,1]^d$ equipped with the Lebesgue measure and the family of axis-parallel boxes,
the minimal dispersion was introduced in the context of uniform distribution theory by Hlawka \cite{Hla76} and
Niederreiter \cite{Nie83}. We refer to \cite{AL24,BC22,Lit21,LL22,TVV24,TVV25} and references therein for history and recent progress.

In the case of $\Sp^d$, the minimal dispersion has been studied for spherical slices (intersections of half-spheres) by Rote and Tichy \cite{RT95} and subsequently by Prochno and Rudolf \cite{PR24} for intersections of two spherical caps, that is for the family
\[
\cL
=\{C_1\cap C_2\colon C_1,C_2\in \cC\},
\]
which we refer to as spherical lenses. The corresponding notion of minimal lens dispersion is then given by
\begin{equation} \label{eq:dispersion-intersect}
\disp_{\cL}(n,d)
:=\inf_{x_1,\dots,x_n\in \Sph} \,\, \sup_{A\in \cL}\, \{\sigma(A)\colon A\cap \{x_1,\dots,x_n\}=\emptyset\}.
\end{equation}

The study of spherical dispersion is motivated by a question of Erd\H{o}s \cite[p.~54]{Erd64}
who asked for a lower bound on the minimal spherical cap discrepancy
\[
D(n,d)
:=\inf_{x_1,\dots,x_n\in \Sp^d}\sup_{C\in \cC} \Big| \frac{|\{i\colon x_i\in C\}|}{n} - \sigma(C)\Big|,
\]
where the supremum is taken over the system of spherical caps $\cC$ (see Blümlinger \cite{Blu91} for the spherical slice discrepancy).
A lower bound for $D(n,d)$ was given by Schmidt \cite{Sch69} and improved by Beck \cite{Bec84b}, who had provided an almost matching
upper bound in \cite{Bec84a}. It is a well-known open problem to determine the asymptotic behavior of $D(n,d)$ as $n\to\infty$ for $d\ge 2$.
Clearly, $D(n,d)\ge \disp_{\cC}(n,d)$ and, as we present below, the behavior of $\disp_{\cC}(n,d)$ is better understood than
that of $D(n,d)$. At the end of the introduction we will comment on $\disp_{\cL}(n,d)$.

Recall that a family $\cC$ of spherical caps coincides with the family of geodesic balls
\[
B(x,\varphi)=\{y\in \Sp^d\colon \varrho(x,y)\le \varphi\},\qquad x\in \Sp^d,\varphi\in [0,\pi],
\]
where $\varrho(x,y)=\arccos (\langle x,y\rangle)$ denotes the geodesic distance.
Note that the dispersion remains unchanged whether we consider open or closed caps,
and we work with the latter.  We start with the observation that the minimal spherical
cap dispersion is related to covering the sphere by caps. Define
the minimal spherical covering density of the sphere by
\[
  \dens(n,d) :=
   \inf\Big\{\sum_{i=1}^{n}\sigma(B(x_i,\varphi))\colon \bigcup_{i=1}^n B(x_i,\varphi)=
   \Sp^d, x_1,\dots,x_n\in \Sp^d, \varphi\in (0,\pi]\Big\}
\]
and the minimal geodesic covering radius  by
\[
\varphi(n)
:=\inf\Big\{\varphi>0\colon \exists x_1,\dots,x_n\in \Sp^d \text{ with }\bigcup_{i=1}^n B(x_i,\varphi)=\Sp^d\Big\}.
\]
These two parameters  have been intensively studied, see the books and surveys \cite{AS17,Bor04,CS99,Tot83, Tot64, Tot72,Rog64}.
We summarize the relations between them and dispersion  in the following lemma, which will be proved in
Section~\ref{subsec:disp-covering}.

\begin{lemma} \label{newlemma} Let $d, n$ be positive integers. Then
\begin{equation} \label{eq:disp-covering-density}
n\cdot\disp_{\cC}(n,d)
=\dens(n,d)
=n\cdot V(\varphi(n)),
\end{equation}
where
  $V(\varphi)=\sigma(B(x,\varphi))$ denotes the normalized volume of a cap with geodesic
  radius $\varphi\in [0,\pi]$  centered at $x\in \Sp^d$.
\end{lemma}

In the following, we derive statements about behavior of the spherical cap dispersion.
It is easy to see that $\disp_{\cC}(n,d)\ge \frac{1}{n}$ (which also follows from $\dens(n,d)\ge 1$ and (\ref{eq:disp-covering-density})).  When $n\leq d+1$ the affine hull of $n$ points spans an affine subspace of $\IR^{d+1}$ with dimension at most $d$. Moreover, any cap of area at least $\frac{1}{2}$ intersects a given pair of antipodal points. Therefore we have
\begin{equation}\label{triv-disp}
\disp_{\cC}(1,d)=1 \quad \quad \mbox{ and } \quad \quad
\disp_{\cC}(2,d)=\cdots=\disp_{\cC}(d+1,d)=\frac{1}{2}.
\end{equation}
For $d+2\le n\le 2d+2$ optimizers are expected to be regular structures, see \cite[Conjecture~1.3]{BW03} and \cite[Chapter~6]{Bor04}.
More precisely, for $n=d+1+k$, $1\leq k\leq d+1$, split $\R^{d+1}$ into an orthogonal sum of $k$ subspaces $E_i$ of the dimensions
$\lceil (d+1)/k\rceil$ and   $\lfloor (d+1)/k\rfloor$ and in each $E_i$ take a regular simplex inscribed into the Euclidean ball.
Then take convex hull of such simplices.  In particular,
for $n=d+2$ (so $k=1$)  a regular simplex inscribed into $\Sp^d$ gives the optimal covering with radius
$\varphi(d+2)=\arccos \frac{1}{d+1}$ and for $n=2d+2$ (so $k=d+1$) the cross-polytope inscribed into $\Sp^d$ yields a covering with
$n$ caps of radius $\arccos \frac{1}{\sqrt{d+1}}$, which is conjectured to be optimal (see also \cite[Problem~4]{DLM+00}). We note also
that such arrangements were used to prove sharpness of the lower bound on the distance between convex polytopes with few vertices
and centrally-symmetric bodies \cite{GL2002}.
It follows from volume bounds given in Section~\ref{subsec:volume} (see Remark~\ref{remone}) that
\begin{equation} \label{eq:simplex-disp}
\frac{1}{2}-\disp_{\cC}(d+2,d)\sim \frac{1}{\sqrt{2\pi d}},
\end{equation}
where for sequences $(a_d)_d$ and $(b_d)_d$ of positive numbers we write $a_d\sim b_d$ for $\lim_{d\to\infty}\frac{a_d}{b_d}=1$. Moreover, if $n=2d+2$, then
\begin{equation} \label{eq:cross-disp}
	\disp_{\cC}(2d+2,d)\le \frac{1+o(1)}{\sqrt{2\pi}}\int_1^{\infty}e^{-x^2/2}\dd x  \qquad \text{as }d\to\infty
\end{equation}
(see Lemma~\ref{lem:vol-gauss} and Remark~\ref{remtwo} below).

\begin{rem}
For completeness we remark that in the case of $d=1$ the sphere $\Sp^1$ is a one-dimensional torus and $\disp_{\cC}(n,1)=\frac{1}{n}$ for every $n\in\IN$. Thus, we focus on $d\ge 2$. We refer to \cite{AL24,Lit21,LL22,Rud18,Ull18} about dispersion with respect to periodic axis-parallel boxes on a torus of dimension $d\ge 2$.
\end{rem}

If $n$ is sufficiently large compared to $d$, then local approximation of $\Sp^d$ by $\IR^d$ suggests that $\dens(n,d)$ is related to the minimal covering density of $\IR^d$ by equally sized balls. The latter can be defined by
\begin{equation} \label{eq:min-covering-density}
\vartheta_d
=\inf_{\mathcal{B}}\lim_{R\to\infty}\sum_{B\in \mathcal{B}} \frac{\vol_d(B\cap B^d(0,R))}{\vol_d(B^d(0,R))},
\end{equation}
where the infimum is taken over any covering $\mathcal{B}=\{B^d(x_1,1),B^d(x_2,1),\dots\}$ of $\IR^d$
by Euclidean balls of unit radius. Indeed, in the next theorem  we show that $\lim_{n\to\infty}\dens(n,d)=\vartheta_d$
(we state all our results in the equivalent notion of dispersion, see Lemma~\ref{newlemma}).

\begin{theorem}\label{thm:disp-covering} Let $d\geq 2$. Then  the minimal dispersion satisfies
\[
\lim_{n\to\infty}n\cdot \disp_{\cC}(n,d)
=\vartheta_d.
\]
\end{theorem}

Instead of proving Theorem~\ref{thm:disp-covering} directly, we deduce it from the known asymptotic for the best
approximation of the Euclidean unit ball $B^{d+1}$ by  inscribed polytopes with $n$ vertices in
Hausdorff distance (see e.g., \cite{Sch81}).

The minimal covering density of $\IR^d$ is known to satisfy $\vartheta_2=\frac{2\pi}{3\sqrt{3}}$ in the case $d=2$ which is attained by the hexagonal lattice (see \cite{Ker39}) and, for $d\ge 3$,
\begin{equation} \label{eq:density-bounds}
\frac{d}{e\sqrt{e}}\sim \tau_d \le \vartheta_d \le d\ln d+d\ln\ln d+5d.
\end{equation}
The lower bound in \eqref{eq:density-bounds} is due to Coxeter-Few-Rogers \cite{CFR59} and the upper is due to Rogers \cite{Rog57},
which was improved by Fejes T\'oth \cite{Tot09} by replacing the summand $5d$ with $d$ and by  Dumer \cite{Dum07} who obtained the bound $(1/2 + o(1))d\ln d$. The constant $\tau_d$ can be written explicitly in terms of the interior angle of a $d$-dimensional regular simplex.

\begin{rem}
	The upper bound in \eqref{eq:density-bounds} also holds for coverings by other convex shapes, while no nontrivial lower bound for shapes other than the ball is known so far, see Fejes T\'oth \cite[p.~36]{GOT17}. We refer to \cite{BGL+25} for recent results developments in this direction.
\end{rem}

It follows from Theorem~\ref{thm:disp-covering} and the bounds in \eqref{eq:density-bounds} that
\begin{equation}
d\lesssim \lim_{n\to\infty}n\cdot \disp_{\cC}(n,d)\lesssim d\ln d,
\end{equation}
where for sequences $(a_d)_d$ and $(b_d)_d$ of positive numbers we write $a_d\lesssim b_d$ if for some $C>0$ and all $d$ one has
$a_d \le C\,b_d$. Note that compared to the minimal dispersion on the cube with respect to axis-parallel boxes (see Bukh and Chao
\cite{BC22}, who showed $cd$ for the lower bound
and $Cd^2 \ln d$ for the upper bound) we  have only a logarithmic  gap between the lower and upper bounds in the case of the spherical cap dispersion.

Regarding the upper bound in \eqref{eq:density-bounds}, Erd\H{o}s and Rogers \cite{ER61} have shown that there exists
a covering of $\IR^d$ by unit balls such that no point is covered by more than $e(d\ln d + d\ln \ln d+5d)$ balls.
In the case of the sphere, Böröczky Jr. and Wintsche \cite{BW03} proved that for any $\varphi<\frac{\pi}{2}$ there is a covering
of $\Sp^d$ by caps of geodesic radius $\varphi$ such that no point of $\Sp^d$ is covered by more than $400d\ln d$ caps.
From  this  one can derive  the following bound.

\begin{prop} \label{pro:400}
Let $d, n\ge 2$.  Then the minimal dispersion satisfies
\begin{equation*}
n\cdot\disp_{\cC}(n,d)\le 400 d\ln d.
\end{equation*}
\end{prop}
In fact, in \cite{BW03} (Remark~5.1 and the end of its proof, see also \cite[Theorem~2.2]{Nas16}) the authors
proved also that there is a covering of the sphere with  the density bounded
by  $d\ln d+d\ln \ln d+5d$ (note, this is the same
 bound as in \eqref{eq:density-bounds}).
 This leads to the following estimate.

\begin{prop}\label{pro:lnlnd}
 Let $d, n\geq 2$.  Then the minimal dispersion satisfies
\[
n\cdot \disp_{\cC}(n,d)
\le d\ln d+d\ln \ln d+5d.
\]
\end{prop}

The bound in Proposition~\ref{pro:lnlnd} can be further improved as follows. Corollary 1.2 in \cite{BW03} states that for every $\varphi\in (0,\arccos\frac{1}{\sqrt{d+1}})$ there is a covering of $\Sp^d$ by at most
\begin{equation} \label{eq:cor-1.2}
C\,\frac{\cos\varphi}{\sin^d\varphi}\cdot d^{3/2}\ln(1+d\cos^2\varphi)
\end{equation}
geodesic balls of radius $\varphi$ with density bounded by $c \ln(1+d\cos^2\varphi)$, where $C, c>0$ are absolute constants.
From this we deduce the following bound which improves Proposition~\ref{pro:lnlnd} if $n$ is subexponential in $d$.

\begin{thm}\label{pro:lnln-upper}
Let $d\geq 2$. Then the minimal dispersion satisfies
\[
n\cdot \disp_{\cC}(n,d)\le C\, d\ln \ln \frac{2n}{d}\qquad  \text{for }\,\,\, n\ge 2d,
\]
where $C>1$ is an absolute constant.
\end{thm}

In particular, Theorem~\ref{pro:lnln-upper} shows that $\disp_{\cC}(cd,d)$ can be made arbitrarily small by choosing $c$ large enough.
The same is true for the spherical cap discrepancy (see \cite[Theorem~3]{HNWW01} which may  be adapted to our situation using
Lemma~\ref{vcdimcap} below). This is in contrast to the related problem of approximation of the Euclidean ball $B^{d+1}$ by inscribed polytopes
with $n$ vertices, where $n$ needs to grow faster than $d^{d/2}$ for the error to converge to zero, which follows from bounds
provided in  \cite{BF88, CP88, Gluskin88} (which are essentially inequality (\ref{eq:tikhomirov}) below).

Theorem~\ref{pro:lnln-upper} and Lemma~\ref{newlemma} imply that $V(\varphi(n))\leq C\, (d/n) \ln \ln (2n/d)$.
This immediately provides a bound on covering numbers of the sphere.
Recall that  a set of points $\mathcal{N}\subset \Sp^d$ is called a (geodesic) $\varepsilon$-net if for every $x\in \Sp^d$ there is
$y\in \mathcal{N}$ such that $\varrho(x,y)\le \varepsilon$. Therefore, adjusting the constant $C$ we have the following bound.

\begin{cor}\label{cor:net}
For every $d\geq 2$ and $\varepsilon \in (0, \pi/2)$ there exists  a geodesic $\varepsilon$-net of $\Sp^d$ with cardinality bounded by
$$
d\, \max\Big\{2, \, \frac{C\,\ln \ln (2/V(\varepsilon))}{V(\varepsilon)}\Big\},
$$
where $C>0$ is an absolute constant.
\end{cor}

This estimate complements standard bounds (see for example Corollary~5.5 in \cite{AS17}) and should be compared with bounds
in Examples~6.1 and 6.2 in \cite{BW03}.

Moreover, the results of \cite{BW03} imply the following lower bound, which unfortunately holds only for $n$ superexponential  in $d$.
This restriction appears as the proof uses the lower bound in \eqref{eq:density-bounds} for $\IR^d$ and a
sufficiently good approximation of $B^{d+1}$ by a polytope with $n$ vertices. We would like to note that a similar
restriction on $n$ also appears (although for a different reason)
in the lower bound in \cite{BC22} for the minimal dispersion in the case of the cube.

\begin{thm}\label{pro:lower}
Let $d\geq 2$.
Then the minimal dispersion satisfies
\[
n\cdot \disp_{\cC}(n,d)
\ge c_0\,d \qquad  \text{for }\,\,\, n\ge C d^{(d+3)/2}\ln d,
\]
where $c_0, C>0$ are absolute constants.
\end{thm}

Note that, asymptotically as $n\to\infty$, the constant $c_0$ in Theorem~\ref{pro:lower} can be chosen as $\frac{1}{e\sqrt{e}}$
 by Theorem~\ref{thm:disp-covering}.

Several upper bounds, including the ones in Propositions~\ref{pro:400}, \ref{pro:lnlnd} and Theorem~\ref{pro:lnln-upper}
(all derived from \cite{BW03}), rely on semi-random constructions of identically and independently distributed (i.i.d.)
random points (with respect to the probability measure $\sigma$) with well-separated points to fill the gaps left by the random points.
Such  a technique goes back at least  to Rogers  \cite{Rog57}, who used it for the upper bound of \eqref{eq:density-bounds}
and for the spherical covering density in \cite{Rog63}. The book \cite[Chapter~5]{AS17}  contains a  presentation
of this technique for the case of covering the sphere by spherical caps. It was also recently used in \cite{AL24} for the minimal dispersion
on the cube. Note that Nasz\'odi \cite{Nas16} used non-probabilistic tools to reprove the bound from \cite{BW03} mentioned before
Proposition~\ref{pro:lnlnd}.

If the number of points $n$ is large compared to the dimension $d$,
then covering of $\Sp^d$ relying only on i.i.d. random centers may introduce an additional logarithmic factor in $n$.
This is because the largest gap between i.i.d. random points is logarithmically larger than the average gap, see \cite{BCL10,Jan86,RS16} for details
(note that in the case $d=1$ it is an easy fact).
This effect is related to the coupon collector's problem (see \cite{HKK+20} for dispersion on the cube).

\begin{rem}
	In a recent work, Hoehner and Kur \cite{HK25} have shown that covering most of the sphere using i.i.d.~random points is
asymptotically optimal in high dimension. More precisely, the maximum proportion of the sphere $\Sp^d$ covered by $n$ spherical
caps of volume $\frac{1}{n}$ tends to $1-\frac{1}{e}$ as $n$ and $d$ tend to infinity.
\end{rem}

 A separate question is what bounds can be obtained by using a random choice of points (with respect to the normalized Lebesgue measure).
 Such a choice of points was used in recent works on dispersion of the cube \cite{Lit21,LL22,Rud18} and for spherical lens dispersion  \cite{PR24}.
 Moreover, the above mentioned bounds in \cite{Rud18} and  \cite{PR24}  are immediate consequences of a result in \cite{BEH+89},
 which essentially uses the so-called
 $\varepsilon$-net theorem in combinatorics due to \cite{HW87} (see Theorem~\ref{thm:disp-vc} below).
 We would like to note that this $\varepsilon$-net theorem has been used by Nasz\'odi \cite{Nas19} to construct a good approximation
 of a convex body by convex polytopes with small number of vertices.
 Below, we use the results from \cite{BEH+89,HW87} to derive the following bound, which is slightly weaker than the one in
 Theorem~\ref{pro:lnln-upper} but relies only on i.i.d.~points.

\begin{prop}\label{pro:upper-vc} Let $d\geq 2$. Then
the minimal dispersion satisfies
\[
n\cdot \disp_{\cC}(n,d)
\le \frac{3(d+2)}{\ln 2}\ln\Big(\frac{2e n}{d+2}\Big) \qquad  \text{for }\,\,\, n\ge d+2.
\]
Moreover, the bound holds for independent  uniformly distributed points on $\Sp^d$.
\end{prop}

In fact, it can be deduced from the proof that the bound holds with  probability $1-\alpha$,
$\alpha\in (0,1)$, for i.i.d. random points if the factor $3$ is replaced by a large enough constant
depending only on $\alpha$. Although the proof of the combinatorial $\varepsilon$-net theorem is not very long
(for the sake of completeness, we present it in Section~\ref{sec:vc} together with the proof of
Proposition~\ref{pro:upper-vc}), there is an even simpler approach using a standard (geometric)
$\varepsilon$-net argument, which however yields a slightly worse bound in the regime
$d+2\leq n \leq Cd^2$ than Proposition~\ref{pro:upper-vc}.

\begin{prop}\label{pro:upper-net}
Let $d, n\ge 2$. Then
the minimal dispersion satisfies
\[
n\cdot \disp_{\cC}(n,d)
\le 12d\ln n.
\]
Moreover, the bound holds for independent  uniformly distributed points on $\Sp^d$.
\end{prop}

The proof of Proposition~\ref{pro:upper-net} will be given in Section~\ref{sec:test-sets} and it is based on the idea of
$\delta$-approximations which were discussed in \cite{LL22} (see also \cite{AL24,Lit21}) for the dispersion on the
cube as a simplification of the notion of $\delta$-covers used in \cite{DRZ19} on the sphere. In the preprint \cite{BPR21},
a predecessor of \cite{PR24}, an approach using $\delta$-covers was carried out, giving bounds of the same order as in Proposition~\ref{pro:upper-net}.
We refer to \cite{Gne08,Gne24,Rud18} for $\delta$-covers and related notions on the cube. Let us note that $\delta$-covers were originally used in discrepancy theory from which the concept of dispersion emerged.

The minimal lens dispersion, i.e. $\disp_{\cL}(n,d)$ defined in \eqref{eq:dispersion-intersect}, is more complicated to analyze since there is no direct connection to covering. Theorem~B in \cite{PR24} asserts that
\begin{equation} \label{eq:prochno-rudolf}
n\cdot \disp_{\cL}(n,d)
\le C d\ln\Big(\frac{en}{32d}\Big)\qquad \emph{for} \,\,\, n\ge 32d,
\end{equation}
where $C>0$ is an absolute constant. Moreover, this holds
 for the expected dispersion of $n$ i.i.d.~random points and the proof is based on bounds  
 derived from \cite{BEH+89} (analogous to our Theorem~\ref{thm:disp-vc}).
 For completeness, we provide a proof of Theorem~\ref{thm:disp-vc} and inequality~(\ref{eq:prochno-rudolf})
 in Section~\ref{sec:vc}. In \cite{BPR21} a suboptimal bound using
 $\delta$-covers was given. We recover this bound with a simpler proof building on $\delta$-approximations
 as used for Proposition~\ref{pro:upper-net}.

\begin{prop}\label{thm:upper-net-intersect}
Let $d, n\ge 2$. Then the minimal lens dispersion satisfies
\[
n\cdot \disp_{\cL}(n,d)
\le 24(d+1)\ln n .
\]
Moreover, the bound holds for independent uniformly distributed points on $\Sp^d$.
\end{prop}

We finally would like to note that the bound  \eqref{eq:prochno-rudolf} and Proposition~\ref{thm:upper-net-intersect} extend to families generated by intersections of an arbitrary  number of caps with essentially the same proof. For future reference we formulate it below and prove it in the last section.

\begin{theorem}\label{newthm}
Let $d, k\geq 2$ and
\[
\cL_k
=\{C_1\cap C_2\cap ...\cap C_k \, \colon\,  C_1, C_2, ..., C_k \in \cC\}.
\]
Let $\disp_{\cL_k}(n,d)$ be defined similarly to $\disp_{\cC}(n,d)$ and $\disp_{\cL}(n,d)$.
Then
\[
n\cdot \disp_{\cL_k}(n,d)
\le C d\, k \ln k  \, \ln\Big(\frac{n}{c d\, k \ln k}\Big)\qquad \emph{for} \,\,\, n\ge C \,d\, k \ln k,
\]
where  $c,C>0$ are absolute constants.
\end{theorem}

The structure of this note is as follows.
In Section~\ref{subsec:volume} we review volume bounds for spherical caps and in Lemma~\ref{lem:e-net}
we provide a bound for the cardinality of a geodesic $\varepsilon$-net on $\Sp^d$.
 Section~\ref{subsec:disp-covering} is devoted to the proof of Lemma~\ref{newlemma}.
  In Section~\ref{subsec:approximation} we present the relation between dispersion and approximation of the ball by polytopes. The subsequent sections contain proofs of theorems and propositions.
Section~\ref{sec:test-sets} contains the proofs of Proposition~\ref{pro:upper-net} and Proposition~\ref{thm:upper-net-intersect}
using standard (geodesic) $\varepsilon$-nets and in Section~\ref{sec:vc} we present proofs of dispersion bounds
(Proposition~\ref{pro:upper-vc}, inequality (\ref{eq:prochno-rudolf}), and Theorem~\ref{newthm}) using VC-dimension.

\section{Preliminaries}\label{sec:proofs}

\subsection{Volume bounds}\label{subsec:volume}

The normalized volume of a spherical cap of geodesic radius $\varphi \in [0,\pi]$ can be computed by
\begin{equation} \label{capformula}
V(\varphi)
=\frac{\vol_{d-1}(\Sp^{d-1})}{\vol_d(\Sp^d)}\int_0^{\varphi}\sin^{d-1}t\dd t
= \frac{\int_0^\varphi \sin^{d-1}t\dd t}{2\int_0^\frac{\pi}{2} \sin^{d-1}t\dd t},
\end{equation}
where the $d$-dimensional volume of $\Sp^d$ is given by
$$\vol_d(\Sp^d) =\frac{2\pi^{(d+1)/2}}{\Gamma(\frac{d+1}{2})}.$$
It is known that the Wallis integral $\int_0^\frac{\pi}{2} \sin^{d-1}t\dd t$ satisfies
\begin{equation} \label{eq:wallis}
\sqrt{\frac{\pi}{2d}}
\le \int_0^{\frac{\pi}{2}}\sin^{d-1}t\,\dd t
\le \sqrt{\frac{\pi}{2(d-1)}}\qquad\text{for}\,\,\, d\ge 2,
\end{equation}
(which follows from \cite[Lemma~1]{BGW82}). We need the following volume bounds.

\begin{lem}\label{lem:volume-bounds}
Let $d\geq 2$.
\begin{enumerate}[label=(\roman*)]
	\item If $\varphi\in (0,\frac{\pi}{2})$, then
		\[
		\frac{1}{\sqrt{2\pi(d+1)}}\sin^d\varphi
		\le V(\varphi)
		\le \frac{1}{2}\sin^d\varphi.
		\]
	\item If $\varphi\le \arccos \frac{1}{\sqrt{d+1}}$, then
		\[
		\frac{1}{3\sqrt{2\pi(d+1)}}\cdot \frac{1}{\cos \varphi} \sin^d\varphi
		\le V(\varphi)
		\le \frac{1}{\sqrt{2\pi d}}\cdot \frac{1}{\cos \varphi} \sin^d\varphi.
		\]
	\item  If $\varphi\ge \arccos \frac{1}{\sqrt{d+1}}$, then
		\[
		\frac{1}{3e\sqrt{2\pi}}
		\le V(\varphi)
		\le \frac{1}{2}.
		\]
	\item If $\alpha\in (0,\frac{\pi}{2})$, then
\[
e^{-\frac{(d-1) \alpha^2}{2\cos^2 \alpha}}\sqrt{\frac{d-1}{2\pi}}\,\alpha
\le \frac{1}{2}-V\Big(\frac{\pi}{2}-\alpha\Big)
\le \sqrt{\frac{d}{2\pi}}\,\alpha.
\]
	\item If $0<\delta \le w \le \pi$, then
		\[
		\frac{V(w)}{V(\delta)}\le \Big(\frac{w}{\delta}\Big)^d.
		\]
\end{enumerate}
\end{lem}
\begin{proof}
The upper bound in (i) is taken from \cite[Proposition~5.1]{AS17}. The lower bound in (i) and
 bounds in (ii) and (iii) are  taken from  \cite[Corollary~3.2]{BW03}.

To prove (iv) write
\begin{equation} \label{eq:iv}
\frac{1}{2}-V\Big(\frac{\pi}{2}-\alpha\Big)
=\frac{1}{2}-\frac{\int_0^{\frac{\pi}{2}-\alpha}\sin^{d-1}t\,\dd t}{2\int_0^{\frac{\pi}{2}}\sin^{d-1}t\,\dd t}
=\frac{\int_{\frac{\pi}{2}-\alpha}^{\frac{\pi}{2}+\alpha}\sin^{d-1}t\,\dd t}{4\int_0^{\frac{\pi}{2}}\sin^{d-1}t\,\dd t}.
\end{equation}
 The numerator in \eqref{eq:iv} can be written as
\[
\int_{\frac{\pi}{2}-\alpha}^{\frac{\pi}{2}+\alpha}\sin^{d-1}t\,\dd t
= \int_{-\alpha}^{\alpha}\cos^{d-1}t\,\dd t
= \int_{-\alpha}^{\alpha}e^{(d-1)\, f(t)}\dd t,
\]
where $f(t)=\ln \cos t$ on $(-\frac{\pi}{2},\frac{\pi}{2})$
satisfies $f(0)=0$, $f'(0)=0$ and $f''(t)=-\frac{1}{\cos^2 t}$. Consequently,
\[
-\frac{1}{\cos^2 \alpha}\frac{t^2}{2} \le f(t) \le -\frac{t^2}{2}\qquad \text{for }t\in (-\alpha,\alpha).
\]
We bound the numerator in \eqref{eq:iv} from above by
\[
\int_{-\alpha}^{\alpha}e^{(d-1)\,f(t)}\dd t
\le \int_{-\alpha}^{\alpha}e^{-(d-1)\frac{t^2}{2}}\dd t
\le 2\alpha,
\]
and from below by
\[
\int_{-\alpha}^{\alpha}e^{(d-1)\,f(t)}\dd t
\ge \int_{-\alpha}^{\alpha}e^{-\frac{d-1}{\cos^2 \alpha}\frac{t^2}{2}}\dd t
\ge 2\alpha e^{-\frac{(d-1)\alpha^2}{2\cos^2 \alpha}}.
\]
Combined with \eqref{eq:wallis} this completes the proof of (iv).

The bound in (v) follows from the Bishop-Gromov volume
comparison theorem (see e.g.~\cite[Theorem~III.4.5]{Cha93}). Note that
it is essentially proven using elementary tools in \cite[Lemma~3.1 (ii)]{BW03} for $0<\delta\le w\le \pi/2$.
For convenience of the reader we extend the proof to $0<\delta\le w\le \pi$.
We will use the following fact.

\begin{claim}
\label{FactSineFunction}
Let $d\geq 2$, $0< \delta \leq \pi$. Then for every $1\leq t\leq \frac{\pi}{\delta}$,
\begin{enumerate}
\item[(i)]
$\sin(t\delta) \leq t\sin(\delta),$
\item[(ii)]
$\delta \sin^{d-1}\delta \leq d\int_0^\delta \sin^{d-1} s\dd s.$
\end{enumerate}
\end{claim}

\begin{proof}
Note that   the function $\frac{\sin x}{x}$ is decreasing on $(0, \pi]$. Since $1\leq t\leq \frac{\pi}{\delta}$,
this implies that
\begin{align*}
\frac{\sin(t\delta)}{t\sin\delta} = \frac{\sin(t\delta)}{t\delta}\cdot\frac{\delta}{\sin\delta} \leq 1,
\end{align*}
which proves (i).

 To prove (ii) consider the function $f(\delta) = d\int_0^\delta \sin^{d-1} s\dd s - \delta\sin^{d-1}\delta$ on $(0,\pi]$.
 Then $f(0)=0$ and
\begin{align*}
f'(\delta) = d\sin^{d-1}\delta - \sin^{d-1}\delta - (d-1)\delta\sin^{d-2}\delta\cos\delta = (d-1)\sin^{d-2}\delta(\sin\delta - \delta\cos\delta).
\end{align*}
Since $\sin\delta - \delta\cos\delta >0$ on $(0,\pi]$, we observe that $f'(\delta)>0$. This proves that $f$ is non-negative
on $(0,\pi]$ and thus proves (ii).
\end{proof}

We continue to prove (v). Setting $t=w/\delta$ we have to show that $V(t\delta) \leq t^{d} V(\delta).$
By (\ref{capformula}) it suffices to show that for every $1\leq t\leq \frac{\pi}{\delta}$,
\begin{align*}
F(t):=  t^{d} \int_0^\delta \sin^{d-1}s\dd s  - \int_0^{t\delta} \sin^{d-1}s\dd s \geq 0.
\end{align*}
Using (i) and (ii) of Claim~\ref{FactSineFunction} we observe that for every $1\leq t\leq \frac{\pi}{\delta}$,
\begin{align*}
\frac{\partial}{\partial t} \left( \int_0^{t\delta} \sin^{d-1}s \dd s\right)
&= \delta\sin^{d-1} (t\delta) \leq t^{d-1}\delta\sin^{d-1}\delta \\
&\leq t^{d-1} d\int_0^\delta \sin^{d-1}s \dd s
= \frac{\partial}{\partial t}\left(t^{d} \int_0^\delta \sin^{d-1}s \dd s \right).
\end{align*}
This shows that $F'\geq 0$ and implies the desired result as $F(1)=0$.
\end{proof}

\begin{remark}\label{remone}
The asymptotic relation \eqref{eq:simplex-disp} for $\disp_{\cC}(d+2,d)$ is an immediate consequence of Lemma~\ref{lem:volume-bounds} (iv) and $\arccos x=\frac{\pi}{2}-x-O(x^3)$ as $x\to 0$.
\end{remark}

If $\frac{\pi}{2}-\varphi$ is of order $\frac{1}{\sqrt{d}}$, then the following lemma gives the asymptotic constant for $V(\varphi)$ as $d\to\infty$.

\begin{lemma} \label{lem:vol-gauss}
For every $\alpha\ge 0$,
\[
V\Big(\frac{\pi}{2}-\frac{\alpha}{\sqrt{d}}\Big)
\sim V\Big(\arccos\frac{\alpha}{\sqrt{d+1}}\Big)
\sim \frac{1}{\sqrt{2\pi}}\int_{\alpha}^{\infty}e^{-x^2/2}\dd x\qquad\text{as } \,\, d\to\infty.
\]
\end{lemma}
\begin{proof}
	The first ``$\sim$'' follows from $\arccos x=\frac{\pi}{2}-x-O(x^3)$ as $x\to 0$. For the second ``$\sim$'' we  interpret $V(\arccos\frac{\alpha}{\sqrt{d+1}})$ as the probability that a uniformly distributed vector on $\Sp^d$ is contained in a cap of radius $\arccos\frac{\alpha}{\sqrt{d+1}}$ which, without loss of generality, is centered at $e_1=(1,0,\dots,0)$. Let $g=(g_1,\dots,g_{d+1})$ be a standard Gaussian vector in $\IR^{d+1}$. Then $\frac{g}{\|g\|_2}$ is uniformly distributed on $\Sp^d$ and
\begin{align*}
\sigma\Big(\Big\{y\in \Sp^d\colon \arccos y_1 \le \arccos \frac{\alpha}{\sqrt{d+1}} \Big\}\Big)
=\IP\Big[\frac{g_1}{\sqrt{S_{d+1}}}\ge \alpha \Big],
\end{align*}
where $S_{d+1}:=\frac{1}{d+1}\sum_{i=2}^{d+1}g_i^2$. By the law of large numbers $S_{d+1}\to 1$ in probability and by the continuous mapping theorem the same holds for $\sqrt{S_{d+1}}$. Then Slutsky's theorem implies that
\[
\IP\Big[\frac{g_1}{\sqrt{S_{d+1}}}\ge \alpha \Big]
\to \IP[g_1\ge \alpha ]
=\frac{1}{\sqrt{2\pi}}\int_{\alpha}^{\infty}e^{-x^2/2}\dd x.
\]
This shows the claimed asymptotic for $V(\arccos\frac{\alpha}{\sqrt{d+1}})$.
\end{proof}

\begin{remark}\label{remtwo}
The asymptotic relation \eqref{eq:cross-disp} for $\disp_{\cC}(2d+2,d)$ directly follows from Lemma~\ref{lem:vol-gauss}.
\end{remark}

We finally provide a standard estimate on the cardinality of geodesic $\varepsilon$-nets on the sphere.
Recall that  a set of points $\mathcal{N}\subset \Sp^d$ forms a geodesic $\varepsilon$-net if for every $x\in \Sp^d$ there is
$y\in \mathcal{N}$ such that $\varrho(x,y)\le \varepsilon$, or in other words, the geodesic covering radius defined in the next subsection satisfies $\varphi(\mathcal{N})\le \varepsilon$. Note that for every $x, y\in \Sp^d$ one has $|x-y|=2\sin (\varrho(x,y)/2)$. The existence of Euclidean or geodesic $\varepsilon$-nets on the sphere $\Sp^d$ of cardinality $(C/\eps)^d$ for some absolute constant $C$ is well-known and follows from the volume argument. We provide a proof with explicit constants  for completeness.  

\begin{lem}\label{lem:e-net}
For every $d\geq 1$ and $\varepsilon\in (0,\pi)$ there is a geodesic $\varepsilon$-net $\mathcal{N}\subset \Sp^d$ of cardinality
  at most $2(\pi/\varepsilon)^d$.
\end{lem}
\begin{proof}
For $d=1$ this is straightforward, so let $d\ge 2$. If $\varepsilon\ge \frac{\pi}{2}$, then choose $\mathcal{N}=\{x,-x\}$ for some $x\in \Sp^d$.

For $\varepsilon\in (0,\frac{\pi}{2})$ we use the standard volume argument. Let $\mathcal{N}\subset \Sp^d$ be a maximal set with $\varrho(x,y)>\varepsilon$ for $x,y\in \mathcal{N}$. Then $\mathcal{N}$ is an $\varepsilon$-net. Since the interiors of balls $B(x,\frac{\varepsilon}{2})$, $x\in \mathcal{N}$, have to be pairwise disjoint, we have
\begin{equation} \label{eq:net}
1=\sigma(\Sp^d)\ge \sum_{x\in \mathcal{N}}\sigma(B(x,\varepsilon/2))=|\mathcal{N}|\, V(\varepsilon/2).
\end{equation}
If $d=2$, then $V(\varepsilon/2)=\frac{1}{2}(1-\cos \frac{\varepsilon}{2})\ge \frac{1}{2}(\frac{\varepsilon}{\pi})^2$, which proves the claim in this case. If $d\ge 3$, we use Lemma~\ref{lem:volume-bounds} (i) and $\sin\frac{\varepsilon}{2}\ge \frac{\sqrt{2}\varepsilon}{\pi}$, $\varepsilon\in (0, \frac{\pi}{2})$, to obtain that
\[
V(\varepsilon/2)
\ge \frac{1}{\sqrt{2\pi(d+1)}}\sin^d(\varepsilon/2)
\ge \sqrt{\frac{2^{d-1}}{\pi(d+1)}}\Big(\frac{\varepsilon}{\pi}\Big)^d
\ge \frac{1}{2}\Big(\frac{\varepsilon}{\pi}\Big)^d.
\]
Combining this with \eqref{eq:net} completes the proof.
\end{proof}

\subsection{Dispersion and covering (proof of Lemma~\ref{newlemma})}
\label{subsec:disp-covering}

Here we prove Lemma~\ref{newlemma}, which provides relations
between minimal dispersion, density, and covering radius, all of which quantify
the maximum efficiency of placing points on the sphere.
Recall that by definition, given a point set $P=\{x_1,\dots,x_n\}\subset\Sph$, its spherical cap dispersion is
\[
\disp_{\cC}(P)
=\sup_{C\in \cC}\{\sigma(C)\colon C\cap P=\emptyset\}.
\]
With a slight abuse of notations we denote the geodesic covering radius by
\[
\varphi(P)
=\inf\Big\{\varphi>0\colon \bigcup_{i=1}^n B(x_i,\varphi)=\Sp^d\Big\}
=\sup\{\varphi>0\colon \exists x\in \Sp^d \text{ with }B(x,\varphi)\cap P=\emptyset\}.
\]
Thus, the geodesic covering radius of $P$ equals the geodesic radius of the largest (open)
ball which does not intersect $P$. Consequently,
\begin{equation} \label{eq:disp-V}
\disp_{\cC}(P)
=\sup\big\{V(\varphi)\colon B(x,\varphi)\cap P=\emptyset, x\in \Sp^d, \varphi\in (0,\pi) \big\}
=V(\varphi(P)).
\end{equation}
Taking the infimum over all $n$-point sets $P_n\subset \Sp^d$ on both sides of this equality and using the continuity of
the function $V(\cdot)$ we observe that
\begin{equation} \label{eq:disp-covering-density-1}
\disp_{\cC}(n,d)
=\inf_{P_n\subset \Sph}\disp_{\cC}(P_n)
=V\big(\inf_{P_n\subset \Sph}\varphi(P_n)\big)
=V(\varphi(n)),
\end{equation}
which proves one part of \eqref{eq:disp-covering-density}.

For the relation to minimal density, note that  definitions of $\dens(n,d)$ and $\varphi(n)$
 immediately imply that for every $\varphi>\varphi(n)$ one has $\dens(n,d)\leq n V(\varphi)$ and
for every $\varphi<\varphi(n)$ one has $\dens(n,d)\geq n V(\varphi)$. These two facts together with
 the continuity of $V(\cdot)$ yield $\dens(n,d)= n V(\varphi(n))$, which completes the proof.
\qed

\subsection{Dispersion and convex body approximation}\label{subsec:approximation}

In this section we present a useful connection between minimal dispersion and best approximation of convex bodies by polytopes. It is well-known that covering $\Sp^d$ by geodesic balls of equal radius is related to approximating the $(d+1)$-dimensional Euclidean unit ball $B^{d+1}$ by the convex hull $\conv(P)$ of the set of centers $P$.
Indeed, it is not difficult to see that the largest ball centered at the origin and contained
in $\conv(P)$ has radius $\cos \varphi(P)$, where $\varphi(P)$ is the geodesic covering radius of $P$ defined above
(see e.g. \cite[Lemma~7.1]{BW03}).  This implies
\begin{equation} \label{eq:hausdorff-cos}
\delta_H(\conv(P),B^{d+1}) =1-\cos \varphi(P),
\end{equation}
where the Hausdorff distance between convex sets $K,L\subset \IR^{d+1}$ is given by
\[
\delta_H(K,L)=\max\Big\{\sup_{x\in K}\inf_{y\in L}\|x-y\|,\,\, \sup_{x\in L}\inf_{y\in K}\|x-y\|\Big\}.
\]
From \eqref{eq:disp-V} and \eqref{eq:hausdorff-cos} we obtain that, for any $n$-point set $P_n\subset \Sp^d$,
\[
\disp_{\cC}(P_n)
=V(\varphi(P_n))
=f(\delta_H(\conv(P_n),B^{d+1})),
\]
with $f(x)=V(\arccos(1-x))$ for $x\in (0,\pi)$. Taking the infimum over all $n$-point sets $P_n\subset\Sp^d$ on both sides and using continuity of $f$ we deduce
\begin{equation} \label{eq:disp-hausdorff}
\disp_{\cC}(n,d)
=f(\delta_H(n,d)),
\end{equation}
where $\delta_H(n,d)$ denotes the error of best approximation of $B^{d+1}$ (in the Hausdorff distance) by polytopes with at most $n$ vertices on $\Sp^d$.

The relation \eqref{eq:disp-hausdorff} allows to transfer bounds for $\delta_H(n,d)$ to minimal dispersion, and vice versa. This will be used below for the proof of Theorem~\ref{thm:disp-covering}. The survey \cite{Bro08} contains many references on the behavior of $\delta_H(n,d)$. Equivalently, one can study the difference $1-\delta_H(n,d)=\cos \varphi(n)$ which is the maximal inradius of a polytope with $n$ vertices inscribed into $\Sp^d$
(see e.g. Lemma 7.1 in \cite{BW03}) and thus related to Banach-Mazur distance between the Euclidean ball and the convex hull $n$ points
on the sphere (see discussion in Section~8 of \cite{BW03}).
In \cite[inequality~(14)]{BW03} lower bounds for $\cos \varphi(n)$ are given for $d+1\le n\le \sqrt{2}^d$.
We will need the following upper bound, which follows from a result on volumes of polytopes with $n\geq  2(d+1)$ vertices on the sphere,
proved independently in  \cite{BF88, CP88, Gluskin88} and a result on the Banach-Mazur distance between the Euclidean ball and
a polytope with $d+2\leq n\leq 2(d+1)$ vertices, proved in \cite{Tik15}. We formulate it in terms of $\varphi(n)$
(cf. similar discussion in  Section~8 of \cite{BW03}). It is known that this bound is sharp.

\begin{theorem} \label{eq:tikhomirov}
Let $n\geq d+2$. Then
$$
\cos \varphi(n) \le C \sqrt{\frac{\ln (n/d)}{d}},
$$
where $C>0$ is an absolute constant.
\end{theorem}

\section{Proofs of statements}
In the following, we provide the proofs of all Theorems and Propositions in order of appearance with four  exceptions. Proposition~\ref{pro:upper-vc} and Theorem~\ref{newthm} are proved in Section~\ref{sec:vc}, while Propositions~\ref{pro:upper-net} and \ref{thm:upper-net-intersect} are proved in Section~\ref{sec:test-sets}.

\begin{proof}[Proof of Theorem~\ref{thm:disp-covering}]
Recall that $\delta_H(n,d)$ denotes the error of best approximation of $B^{d+1}$
 by polytopes with at most $n$ vertices on $\Sp^d$ and that by \eqref{eq:disp-hausdorff},
\[
\disp_{\cC}(n,d)=V(\arccos(1-\delta_H(n,d))).
\]
On the other hand it is well known that
\begin{equation} \label{eq:best-app}
\delta_{H}(n,d)
=\frac{1+o(1)}{2}\Big(\frac{\vol_d(\Sp^d)}{\vol_d(B^d)}\vartheta_{d}\Big)^{2/d}n^{-2/d}\qquad\text{as }n\to\infty,
\end{equation}
where $\vartheta_d$ is the minimal covering density given in \eqref{eq:min-covering-density} (see e.g. surveys \cite[eq.~(4) and (5)]{Gruber93} and
 \cite[eq.~(7)]{Bro08} and references therein, note that this result goes back to \cite{Sch81}).
Combining the asymptotics $\arccos(1-x)=\sqrt{2x}+O(x^{3/2})$ as $x\to 0$ and
\[
V(\varphi)
=\frac{d\vol_d(B^d)}{\vol_d(\Sp^d)}\int_0^{\varphi}\sin^{d-1}t\dd t
= (1+o(1))\frac{\vol_d(B^d)}{\vol_d(\Sp^d)}\varphi^d \qquad \text{as }\varphi\to 0
\]
 with \eqref{eq:best-app}, we obtain
\[
\lim_{n\to\infty}n\cdot\disp_{\cC}(n,d)
= \vartheta_d \qquad \text{as }n\to\infty,
\]
as claimed.
\end{proof}

\begin{rem}
Note that in \cite[eq.~(13)]{BW03} a slightly less precise asymptotic for $\delta_H(n,d)$
than the one in \eqref{eq:best-app} is derived from the covering bound  \eqref{eq:cor-1.2}.
\end{rem}

In  the next three proofs we will use the following lemma. Recall that $V(\varphi)$ denotes the normalized Lebesgue measure of a cap of geodesic radius $\varphi$ and, given a cover $\mathcal{B} =\{B_1, B_2, ..., B_m\}$ of the sphere $\Sp^d$ the density of this cover is $\sum_{i=1}^m \sigma (B_i)$. When the cover $\mathcal{B}$ contains only spherical caps of the same radius $\varphi$, we say that $\mathcal{B}$ is a $\varphi$-cover and denote its density by $\dens (\varphi,\mathcal{B})$. Clearly, $\dens (\varphi,\mathcal{B})=m V(\varphi)$.

\begin{lemma} \label{newdens}
	Let $d\geq 2$, $\varphi_0\in (0, \pi/2]$, and let $D: (0, \varphi_0]\to [1, \infty)$ be a continuous decreasing function. Assume that for every $\varphi \in  (0, \varphi_0]$ there exists a $\varphi$-cover $\mathcal{B}$ of the sphere with density $\dens (\varphi,\mathcal{B})\leq D(\varphi)$.
 Then for every $n\geq D(\varphi_0)/V(\varphi_0)$ one has $n \disp_{\cC}(n,d)\leq D(\varphi(n)).$
\end{lemma}

\begin{proof}
Let $n\geq D(\varphi_0)/V(\varphi_0)$. Since $1/V$ and $D$ are continuous decreasing functions with $D\geq 1$ and
$V(\varphi)\to 0$ as $\varphi \to 0$, there exists a
$\psi_n  \in  (0, \varphi_0]$ such that $n= D(\psi_n)/V(\psi_n)$. By the definition of  $\varphi$-cover and by the
assumptions, there exists $\psi_n$-cover $\mathcal{B}$ of  the sphere by $m$ spherical caps of radius $\psi_n$ with the density
$$
m V(\psi_n) = \dens (\psi_n,\mathcal{B}) \leq D(\psi_n) = n V(\psi_n).
$$
Thus, $m\leq n$ so we also have a cover of the sphere by $n$  spherical caps of radius $\psi_n$ (just by adding caps
if needed). This implies that $\varphi(n)\leq \psi_n$. Therefore, using Lemma~\ref{newlemma}, we obtain
$$
  n \disp_{\cC}(n,d) = n V(\varphi(n)) \leq  n V(\psi_n) = D(\psi_n)\leq D(\varphi(n)),
$$
which completes the proof.
\end{proof}

\begin{proof}[Proof of Proposition~\ref{pro:400}]
By \cite[Theorem~1.1]{BW03}, for every $\varphi\le \frac{\pi}{2}$,
there exist centers $x_1,\dots,x_{m}\in \Sp^d$ such that every point of $\Sp^d$
is contained in at least one of the balls in $\mathcal{B}=\{B(x_1,\varphi),\dots,B(x_{m},\varphi)\}$
and in at most $400 d\ln d$ many. In \cite{BW03} this is stated for $d\ge 3$
but the proof works also for $d=2$. This implies that for every $\varphi\le \frac{\pi}{2}$
\begin{align*}
	\dens (\varphi,\mathcal{B}) = m\cdot V(\varphi)
&=\sum_{i=1}^{m}\int_{\Sp^d}1_{B(x_i,\varphi)}(y)\dd\sigma(y)
\\&=
\int_{\Sp^d}\sum_{i=1}^{m}1_{B(x_i,\varphi)}(y)\dd\sigma(y)
\le 400d\ln d.
\end{align*}
Applying Lemma~\ref{newdens} with the constant function
$D(\varphi)=400d\ln d$ on $(0, \pi/2]$, we obtain
the desired bound for
$n\geq \frac{400d\ln d}{V(\pi/2)}=800d\ln d.$ The case
$ n< 800d\ln d$  follows from the trivial bounds  $\disp_{\cC}(n,d)\le \frac{1}{2}$ ($n\geq 2$) 
and $\disp_{\cC}(1,d)=1$ (cf., (\ref{triv-disp})). 
\end{proof}

\begin{proof}[Proof of Proposition~\ref{pro:lnlnd}]
Denote $A_d=d\ln d+d\ln \ln d+5d$.
Remark~5.1 and the end of its proof in \cite{BW03} (see also \cite[Theorem~2.2]{Nas16})
imply that for every $\varphi\le \frac{\pi}{2}$ there is a $\varphi$-cover $\mathcal{B}$ with $\dens (\varphi,\mathcal{B})\leq A_d$.  Similarly to the proof of Proposition~\ref{pro:400},
using Lemma~\ref{newdens} with the constant function $D(\varphi)=A_d$ for $n\geq 2 A_d$
and $\disp_{\cC}(n,d)\le \frac{1}{2}$ for $2\le n< 2 A_d$, $\disp_{\cC}(1,d)=1$ we complete the proof.
\end{proof}

\begin{proof}[Proof of Theorem~\ref{pro:lnln-upper}]
Let $\varphi_0:=\arccos \frac{1}{\sqrt{d+1}}$. Corollary~1.2 in \cite{BW03} states that
$$\dens(\varphi) \leq c d \ln(1+d\cos^2\varphi):=D(\varphi)$$ on $(0, \varphi_0]$,
where $c>0$ is an absolute constant. Note that $D(\varphi_0)\leq c d \ln 2$ and
that  by   Lemma~\ref{lem:volume-bounds} (iii),
$$
  \frac{1}{3e\sqrt{2\pi}}<V(\varphi_0)<\frac{1}{2}.
$$
Applying Lemma~\ref{newdens} and Theorem~\ref{eq:tikhomirov} we obtain for every
$n\geq c_1 d$,
$$
   n \disp_{\cC}(n,d)\leq c d \ln(1+d\cos^2(\varphi(n)))\leq
    \frac{c\,d}{n} \ln\Big(1+C^2\ln \frac{n}{d}\Big)
\leq  \frac{c_2\,d}{n} \ln \ln \frac{n}{d},
$$
where $C, c_1, c_2\geq 2$ are absolute constants.
This completes the proof for $n\ge c_1\, d$. For $2d\le n\le c_1\, d$
 we use the trivial bound
\[
\disp_{\cC}(n,d)
\le 1
\le \frac{C_2\,d}{n} \ln \ln \frac{2n}{d},
\]
for a large enough absolute constant $C_2$.
\end{proof}

\begin{proof}[Proof of Corollary~\ref{cor:net}]
	Let $\varepsilon\in (0,\pi/2)$ and let $n$ be the minimal cardinality of a geodesic $\varepsilon$-net. 
Note that  $V(\varepsilon)\le 1/2$. If $n\leq 2d$ we are done, so we assume $m:=n-1\geq 2d$. 
By the definition of the minimal covering radius we have 
$V(\varphi(n))\le V(\varepsilon)\le V(\varphi(m))$, therefore Theorem~\ref{pro:lnln-upper} and Lemma~\ref{newlemma} 
yield 
\[
\frac{m}{d}= \frac{m \disp_{\cC}(m,d)}{d\, V(\varphi(m))}
\le\frac{C\ln\ln(2m/d)}{V(\varphi(m))}
\le\frac{C\ln\ln(2m/d)}{V(\varepsilon)}.
\]
This implies 
\[
\frac{m}{d}\le C_1\frac{\ln\ln(2/V(\varepsilon))}{V(\varepsilon)}
\]
for a large enough absolute constant $C_1>0$. Since $n\leq 2m$, the desired bound follows.
\end{proof}

\begin{proof}[Proof of Theorem~\ref{pro:lower}]
The following proof is based on \cite[Example~6.3]{BW03} where the lower bound
by Coxeter-Few-Rogers  \cite{CFR59} (the left hand side of \eqref{eq:density-bounds})
is adapted to the sphere. More precisely, in \cite[Example~6.3]{BW03} it was shown that
for any finite covering $\cB$ of $\Sp^d$ by caps of equal geodesic radius
$\varphi\le \arcsin \sqrt{\frac{1}{d+1}}$ the density satisfies $\dens(\cB)\ge c\, d$
with an absolute constant $c>0$ (although not explicitly stated in \cite{BW03},
the case $d=2$ is also covered there). Therefore Lemma~\ref{newlemma} implies that
\begin{equation} \label{eq:lower}
n\cdot V(\varphi(n)) = n\cdot\disp_{\cC}(n,d)=\dens(n,d)
\ge c \, d,
\end{equation}
whenever $n\ge n_d$ for $n_d$ satisfying $\varphi(n_d)\le \arcsin\frac{1}{\sqrt{d+1}}$, that is,
\begin{equation}\label{sinpfi}
 \sin \varphi(n_d)\le \frac{1}{\sqrt{d+1}}.
\end{equation}
  By the volume bound in Lemma~\ref{lem:volume-bounds} (i) and Proposition~\ref{pro:lnlnd} we have
\[
\frac{1}{\sqrt{2\pi(d+1)}}\sin^d\varphi(n_d) \le V(\varphi(n_d))\le \frac{7 d\ln d}{n_d}.
\]
Thus, if  $n_d\ge C d^{(3+d)/2}\ln d$ for some suitable absolute constant $C>0$, then (\ref{sinpfi})
holds, which completes the proof.
\end{proof}

\section{Dispersion bounds via geodesic $\varepsilon$-nets}\label{sec:test-sets}

In this section we provide the proofs of Propositions~\ref{pro:upper-net} and \ref{thm:upper-net-intersect}. For convenience, we mainly use $|\cdot|$ instead of $\sigma(\cdot)$ for the normalized spherical volume and the inverse of the minimal spherical dispersion which is given for $\gamma\in (0,1)$ by
\[
N_{\cM}(\gamma,d):=\min\{n\geq 1 \colon \disp_{\cM}(n,d)\le \gamma\},
\]
where $\cM$ is either $\cC$ or $\cL$, so we use either $N_{\cC}$ or $N_{\cL}$
(but in fact it can be any other family of measurable subsets of $\Sp^d$).

Similarly, for each $\gamma \in (0,1)$ consider
$$\mathcal{M}_\gamma := \{ A \in \mathcal{M}: |A| = \gamma\},$$
where $\cM$ is either $\cC$ or $\cL$, so we use either $\cC_\gamma$ or $\cL_\gamma$.
As in \cite{AL24,Lit21} we call a finite family $\mathcal{N} \subset \mathcal{M}$ a $(c_0 \gamma)$-approximation for $\mathcal{M}_\gamma$ with $c_0 \in (0,1)$, if for every $A \in \mathcal{M}_\gamma$ there exists $B \in \mathcal{N}$ such that $B \subset A$ and $|B| \geq c_0\gamma$. Repeating the proof of
 \cite[Lemma~2.3]{Lit21} (see also \cite[Theorem~1]{Rud18}) in our setting we obtain the following.

\begin{lemma} \label{lemma}
Let $d\geq 2$ and let  $\cM$ be either $\cC$ or $\cL$. Let $c_0\in (0,1)$ and $\gamma \in (0,1)$. If $\mathcal{N}\subset \mathcal{M}$ is a $(c_0\gamma)$-approximation for $\mathcal{M}_{\gamma}$ and if $|\mathcal{N}|\geq 3$, then
	$$N_{\cM}(\gamma, d) \leq \frac{3\ln |\mathcal{N}|}{c_0\gamma}.$$
Moreover, the result holds for the random choice of points.
\end{lemma}

It is worth mentioning that recently Arman and Litvak in \cite[Lemma~3.3]{AL24} proved a variation of the above result,
which again transfers to our setting.
\begin{lemma}
\label{lemma2}
Let  $d\geq 2$ and let  $\cM$ be either $\cC$ or $\cL$.
Let $c_0\in (0,1)$ and $\gamma \in (0,\frac{1}{3c_0})$. If $\mathcal{N}\subset \mathcal{M}$ is a $(c_0\gamma)$-approximation for $\mathcal{M}_{\gamma}$ and if $|\mathcal{N}|\geq \frac{e}{c_0 \gamma}$, then
$$N_{\cM}(\gamma, d) \leq \frac{\ln(4c_0\gamma|\mathcal{N}|)}{c_0\gamma}.$$
\end{lemma}

 We would like to emphasize that the proof of Lemma \ref{lemma} is based on a random choice of points,
while the proof of Lemma \ref{lemma2} introduces an additional non-random step.

In order to estimate $N_{\cM}(\gamma,d)$ from above for some $\gamma\in (0,1)$, one can construct a
$(c_0\gamma)$-approximation with small cardinality and $c_0=\frac{1}{2}$, say. We provide a quick summary
of a construction of such an approximation $\mathcal{N}$ for $\mathcal{M}_\gamma$. Since every cap
$C(v,\delta)$ from $\mathcal{M}_\gamma$ is uniquely determined by its center $v$, we start with an
$\varepsilon$-net $\mathcal{N}_0$ for $\Sp^d$ and consider $\mathcal{N}$ to be the family of caps centered at
 points from $\mathcal{N}_0$ with the radius $\delta - \varepsilon$. This guarantees that every
 cap from $\mathcal{M}_\gamma$ contains a member of $\mathcal{N}$. Moreover, by choosing $\varepsilon$
 sufficiently small we ensure that their volumes are comparable.

As a tool, we  need the following lemma which will be applied first for the family $\cC$ and later on for $\cL$.

\begin{lem}\label{lem:volume-radius}
	Let  $d\geq 2$,   $\gamma\in (0,1)$ and $A=B(v_1, \delta_1) \cap B(v_2, \delta_2)$ with some $v_1,v_2\in \Sp^d$.
Let  $\delta_1,\delta_2\in (0,\pi]$ be such that $|A|=\gamma$. Then $\delta=\min\{\delta_1,\delta_2\}\ge \gamma$.
\end{lem}
\begin{proof}
Assuming $\delta\le\frac{\pi}{2}$	and using the bound in Lemma~\ref{lem:volume-bounds} (i) we get
\[
\gamma
=|A|
\le V(\delta)
\le \frac{\sin^d \delta}{2}
\le \frac{\delta^d}{2}.
\]
Therefore,
\[
\delta\ge \gamma^{1/d} \ge  \gamma.
\]
If $\delta>\frac{\pi}{2}$, then $\delta>1\geq \gamma$, which completes the proof.
\end{proof}

To prove Proposition~\ref{pro:upper-net} we construct a $\frac{\gamma}{2}$-approximation for the family $\cC$.

\begin{lem}\label{lem:test-caps}
Let	 $d\geq 2$. For every $\gamma\in (0,1)$ there exists a $\frac{\gamma}{2}$-approximation
$\mathcal{N}$ for $\cC_{\gamma}$ with cardinality
\[
|\mathcal{N}| \le 2\Big(\frac{3d\pi}{\gamma^2}\Big)^{d}.
\]
\end{lem}
\begin{proof}
	Let $\gamma\in (0,1)$ and choose $\delta\in (0,\pi)$ such that $V(\delta)=\gamma$. Let $\varepsilon\in (0,\delta)$ and by Lemma~\ref{lem:e-net} choose an $\varepsilon$-net $\cN_0\subset \Sp^d$ of cardinality at most $2(\frac{\pi}{\varepsilon})^d$. Then, for every $B(v,\delta)\in \cC_{\gamma}$ we find $w\in \cN_0$ such that $\varrho(v,w)\le \varepsilon$ and thus $B(w,\delta-\varepsilon)\subset B(v,\delta)$. Let
\[
\cN=\{B(w,\delta-\varepsilon)\colon w\in \cN_0\}.
\]
Then Lemma~\ref{lem:volume-bounds} (v) implies that
$$\frac{|B(v,\delta)|}{|B(w,\delta-\varepsilon)|} \leq \Big( \frac{\delta}{\delta-\varepsilon} \Big)^d.$$
Choosing $\varepsilon = \frac{\gamma\delta}{3 d}$, we observe that
$$\frac{|B(v,\delta)|}{|B(w,\delta-\varepsilon)|}
\leq \Big( \frac{1}{1-\frac{\gamma}{3d}} \Big)^d
=\Big( 1+\frac{\gamma}{3d(1-\frac{\gamma}{3d})} \Big)^d
\le \Big( 1+\frac{\gamma}{2d} \Big)^d
\le e^{\frac{\gamma}{2}}.$$
 This implies
$$|B(w,\delta-\varepsilon)| \geq \gamma e^{-\frac{\gamma}{2}}\ge \frac{\gamma}{2}.$$
Thus the family $\mathcal{N}$ is a $\frac{\gamma}{2}$-approximation for $\cC_{\gamma}$ of cardinality at most
$$|\mathcal{N}_0|\le 2\Big(\frac{\pi}{\varepsilon}\Big)^{d}
= 2\Big(\frac{3d\pi}{\delta\gamma}\Big)^{d}
\le  2\Big(\frac{3d\pi}{\gamma^2}\Big)^{d},
$$
where we used that $\delta\ge \gamma$ by Lemma~\ref{lem:volume-radius}.
\end{proof}

\begin{proof}[Proof of Proposition~\ref{pro:upper-net}]
 Set $\gamma=\frac{12d}{n}\ln n$. If $\gamma\geq 1$, then $\disp_{\cC}(n,d)\le \gamma$ trivially holds
and we are done. Assume $\gamma<1$. Then, using Lemmas~\ref{lemma} and \ref{lem:test-caps}, we have
\begin{equation*}
N_{\cC}(\gamma, d) \leq \frac{6\ln|\mathcal{N}|}{\gamma} \leq \frac{6\ln\Big(2\big(\frac{3\pi d}{\gamma^2}\big)^{d}\Big)}{\gamma}\\
\leq  \frac{6d\ln\Big(\frac{5\pi d}{\gamma^2}\Big)}{\gamma}\le \frac{n}{2}\cdot\frac{\ln(\frac{n^2}{d})}{\ln n}
\le n.
\end{equation*}
This implies $\disp_{\cC}(n,d)\le \gamma =\frac{12d}{n}\ln n$ as required.
\end{proof}

Next we construct an approximation for $\mathcal{L}_\gamma$.
The construction that will be used  is a natural extension of the one we used for caps in Lemma~\ref{lem:test-caps}, with the only difference
that now one has to control not only centers of caps, but also their radii. First we observe that by Lemma~\ref{lem:volume-radius}
caps forming an intersection of volume $\gamma$ have radii from the interval $(\gamma,\pi]$. We discretize this interval to approximate radii.
Then we construct a finite family of intersections $\mathcal{N}$ using the centers from a given $\varepsilon$-net
and radii from the first step which ensures that every lens of volume $\gamma$ contains a member of the family $\mathcal{N}$.
In the final step by choosing $\varepsilon$ in an appropriate way we ensure that the volume loss is sufficiently small.

\begin{lemma}
\label{gamma}
Let $d\geq 2$.
For every $\gamma \in (0,1)$ there exists a $\frac{\gamma}{2}$-approximation $\mathcal{N}$ for $\mathcal{L}_\gamma$ with cardinality
\[
|\mathcal{N}| \le 9\Big(\frac{12d\pi}{\gamma^2}\Big)^{2(d+1)}.
\]
\end{lemma}
\begin{proof}
	Fix $\gamma \in (0,1)$. Let $\varepsilon= \frac{\gamma^2}{12 d}$.
 By Lemma~\ref{lem:e-net} there exists an $\varepsilon$-net $\mathcal{N}_0\subset\Sp^d$ of cardinality at most $2\big(\frac{\pi}{\varepsilon}\big)^{d}$.
 By Lemma~\ref{lem:volume-radius}, we have $\delta_1,\delta_2\ge \gamma$ and we cover the interval $(\gamma, \pi]$ by sufficiently small subintervals
 in the following way:
$$(\gamma, \pi] \subset \bigcup_{i=1}^k \Big(\pi-i\varepsilon, \hspace{0.2cm} \pi - (i-1)\varepsilon\Big],$$
where $k \in \N$ is the smallest integer such that $\pi-k\varepsilon\leq \gamma$, i.e.,
$$k = \Big\lceil \frac{\pi-\gamma}{\varepsilon} \Big\rceil \leq \frac{\pi}{\varepsilon} + 1.$$
Consider now the family
\begin{align*}
\mathcal{N}=\{B(w_1, \pi-i\varepsilon) \cap B(w_2, \pi-j\varepsilon): w_1,w_2 \in \mathcal{N}_0, i,j=1,2,...,k+1\}.
\end{align*}
We show that $\mathcal{N}$ is  $\frac{\gamma}{2}$-approximation for  $\mathcal{L}_\gamma$ of required cardinality.
We first  estimate the cardinality. Since  $\varepsilon \leq \frac{\pi}{4}$,
\begin{align*}
|\mathcal{N}_0| \cdot (k+1) \leq 2\left(\frac{\pi}{\varepsilon}\right)^{d} \cdot \left(\frac{\pi}{\varepsilon}+2\right)
\le 3\left(\frac{\pi}{\varepsilon}\right)^{d+1}.
\end{align*}
Therefore,
$$|\mathcal{N}|= |\mathcal{N}_0\times \mathcal{N}_0| \cdot (k+1)^2 \leq  9\left(\frac{\pi}{\varepsilon}\right)^{2(d+1)} = 9\left(\frac{12\pi d}{\gamma^2}\right)^{2(d+1)}.$$

Next we prove that $\mathcal{N}$ is  $\frac{\gamma}{2}$-approximation.
Let $A=B(v_1,\delta_1)\cap B(v_2,\delta_2) \in \mathcal{L}_\gamma$.
 Then $i=1,2$ there exists $\ell_i \leq k$  such that
$$\delta_i \in \Big(\pi-\ell_i\varepsilon, \pi - (\ell_i-1)\varepsilon\Big]$$
and  there exist $w_i \in \mathcal{N}_0$ such that
$$B(w_i, \pi-(\ell_i+1)\varepsilon) \subset B(v_i,\pi-\ell_i\varepsilon)\subset B(v_i,\delta_i).$$
Since $\delta_i\ge \gamma$ by Lemma~\ref{lem:volume-radius}, we have
\begin{equation} \label{eq:pi-ell}
\pi-(\ell_i+1)\varepsilon\ge \delta_i -2\varepsilon =\delta_i - \frac{\gamma^2}{6 d} \ge \delta_i\Big(1-\frac{\gamma}{6 d}\Big)>0
\end{equation}
for $i=1,2$. Note that
$$B:=B(w_1,  \pi-(\ell_1+1)\varepsilon) \cap B(w_2, \pi-(\ell_2+1)\varepsilon) \in \mathcal{N}$$
 and $B \subset A$. Therefore,
\begin{align}
\label{eq11}
|B| \geq |A| - |B(v_1, \delta_1) \setminus B(w_1, \pi-(\ell_1+1)\varepsilon)| - |B(v_2, \delta_2) \setminus B(w_2, \pi-(\ell_2+1)\varepsilon)|.
\end{align}
 Using Lemma~\ref{lem:volume-bounds} (v) and \eqref{eq:pi-ell} we observe
\begin{align}
\label{alpha1}
\frac{|B(v_i, \delta_i)|}{|B(w_i, \pi-(\ell_i+1)\varepsilon)|}
\le  \Big( \frac{1}{1-\frac{\gamma}{6 d}} \Big)^{d}
=\Big( 1+\frac{\gamma}{d(6-\frac{\gamma}{d})} \Big)^{d}
\le \Big( 1+\frac{\gamma}{5d} \Big)^{d}
\le e^{\gamma/5}.
\end{align}
Since
$$|B(v_i, \delta_i) \setminus B(w_i, \pi-(\ell_i+1)\varepsilon)| = |B(v_i, \delta_i)| - |B(w_i, \pi-(\ell_i+1)\varepsilon)|,$$
 and since $|B(w_i, \pi-(\ell_i+1)\varepsilon)|\leq 1$, formula (\ref{alpha1}) implies
\begin{align*}
	|B(v_i,\delta_i) \setminus B(w_i, \pi-\ell_i\varepsilon)| \le e^{\gamma/5} - 1.
\end{align*}
Combining the above estimates with (\ref{eq11}) and using $|A|=\gamma$, and $e^{\gamma/5}\le 1+\frac{\gamma}{4}$ for $\gamma\in (0,1)$,
we obtain
\begin{align*}
	|B| \geq \gamma - 2(e^{\gamma/5}-1)
	\ge \frac{\gamma}{2}.
\end{align*}
 As $B\in \mathcal{N}$ and $B \subset A$, this completes the proof.
\end{proof}

\begin{proof}[Proof of Proposition~\ref{thm:upper-net-intersect}]
We proceed as in the proof of Proposition~\ref{pro:upper-net}. Set 
 $\gamma=\frac{24(d+1)}{n}\ln n$. If $\gamma\geq 1$, then clearly $\disp_{\cC}(n,d)\le 1\leq \gamma$. 
 Assume $\gamma<1$.
 By Lemma~\ref{gamma},  there exists
$\frac{\gamma}{2}$-approximation $\mathcal{N}$ for  $\mathcal{L}_\gamma$ of cardinality at most
$9(\frac{12\pi d}{\gamma^2})^{2(d+1)}$.
 Applying this with Lemma \ref{lemma} we obtain
\begin{equation}\label{eq:inverse-lens}
N_{\cL}(\gamma, d) 
\leq \frac{6\ln\Big(9\big(\frac{12\pi d}{\gamma^2}\big)^{2(d+1)}\Big)}{\gamma}
\le \frac{12(d+1)}{\gamma}\ln\Big(\frac{18\pi d}{\gamma^2}\Big)\le 
\frac{n}{2}\cdot\frac{\ln(\frac{n^2}{d})}{\ln n}\le n.
\end{equation}
This implies the desired bound.
\end{proof}

\begin{remark}
	In particular, when $\gamma < \frac{1}{d}$  we get
	$$N_{\cL}(\gamma, d)
\leq \frac{6(d+1)}{\gamma}\ln(18\pi d) + \frac{12(d+1)}{\gamma}\ln\Big(\frac{1}{\gamma}\Big)
\le \frac{18(d+1)}{\gamma}\ln\Big(\frac{18\pi}{\gamma}\Big),$$
which is of the same order as the bound given in \cite{PR24} using the combinatorial $\varepsilon$-net theorem.
\end{remark}

\begin{remark}
	By combining Lemma \ref{gamma} with Lemma \ref{lemma2} instead of Lemma \ref{lemma} we can improve the absolute constants in \eqref{eq:inverse-lens} but at the price of not having random choice of points on the sphere.
\end{remark}

\section{Dispersion bounds via $\varepsilon$-traversals}\label{sec:vc}

In this section we obtain Proposition~\ref{pro:upper-vc}  as a consequence of the combinatorial $\varepsilon$-net theorem.
To phrase it, let $\cX$ be a nonempty set equipped with a $\sigma$-algebra $\Sigma$. Moreover, let $\cR\subset \Sigma$
be a universally separable set system of measurable sets, meaning that there exists a countable subsystem $\cR_0\subset \cR$
with the following property: {\it each $R\in \cR$ can be approximated by a sequence $R_1,R_2,\dots$ in $\cR_0$ in the sense that
$x\in R$ if and only if $x\in R_i$ for all but finitely many indices $i$.} Note that the assumptions are satisfied in the case of $\Sp^d$
equipped with its Borel $\sigma$-algebra and for the choice of $\cR=\cC$ or $\cR=\cL_k$ since each cap can be approximated by
caps with rational centers and radii.

A suitable notion of dimension in the general framework was given by Vapnik–-Chervonenkis in \cite{VC15}, which is nowadays called  the VC-dimension. To introduce it,
first define the shatter function of $\cR$ by
\begin{equation}\label{shatter}
\Pi_{\cR}(k)=\max_{\substack{P_k\in \cX\\ |P_k|=k}}|\{P_k\cap R\colon R\in \cR\}|,
\end{equation}
which returns the maximal number of subsets of a $k$-subset of $\cX$ which arise as intersections with $R\in\cR$. The VC-dimension of $\cR$, given by
\[
\vdim(\cR)=\sup\{k\in\IN_0\colon \Pi_{\cR}(k)=2^k\},
\]
that is, the maximal cardinality of a set such that all its subsets can be realized as intersections with $R\in \cR$. The \textit{dispersion} of $P\subset\cX$ with respect to $\cR$ and a probability measure $\mu$ on $(X,\Sigma)$ is
\[
\disp(P,\cR,\mu)=\sup_{R\in \cR\colon R\cap P=\emptyset}\mu(R).
\]

The following probabilistic upper bound on the dispersion is essentially contained in \cite[Appendix~A]{BEH+89}.
Since  other notation, terminology, and concepts are used there, for the reader's convenience,
we present a proof at the end of this section.

\begin{thm}\label{thm:disp-vc}
Let $\mu$ be a probability measure on $(X,\Sigma)$. Assume that $d=\vdim(\cR)$ is finite and let $m\ge d$. Suppose that $X_1,\dots,X_m$ are distributed independently according to $\mu$. Then for all $\varepsilon>0$ one has
\[
\IP\Big[\disp(\{X_1,\dots,X_m\},\cR,\mu)>\varepsilon\Big]
\le 2\Big(\frac{2e m}{d}\Big)^{d} 2^{-\varepsilon m/2}.
\]
\end{thm}

Note that if for a given $m\ge d$ we set $\varepsilon=\varepsilon_m=\frac{3}{\ln 2}\frac{d}{m}\ln(\frac{2e m}{d})$, then
the probability in   Theorem~\ref{thm:disp-vc} becomes strictly less than 1, which leads to existence of
 an $m$-point set on $\cX$ with dispersion at most $\varepsilon_m$. Thus, we only need to have good bounds on
 the VC-dimension of the class under consideration. For the class of caps it is given by the following lemma.

\begin{lem}\label{vcdimcap}
Let $d\geq 2$.
The class $\mathcal{C}=\mathcal{C}_d$ of caps on $\Sp^d\subset \IR^{d+1}$ satisfies $\vdim(\mathcal{C}_d)=d+2$.
\end{lem}

 For a short proof of Lemma~\ref{vcdimcap} we refer to \cite[Prop.~5.12]{BL15} (note that there  the notation
 $\Sp^n$ is used for $(n-1)$-dimensional sphere in $\R^n$ and that this fact for $d=2$ was proved
 in  \cite[Prop.~8]{ABD12}, where the authors used the definition of VC-dimension which differs by 1
 from the standard one).
 To obtain an upper bound on the dispersion
 in the case of intersections of caps one can use Lemma~3.2.3 in \cite[Lemma~3.2.3]{BEH+89},
 which provides an estimate on the VC-dimension (it implies that the class $\cL_k$, consisting of intersections
 of  at most $k$ spherical caps, has VC-dimension at most $2 (d+2) k \log_2 (3k)$).

\begin{proof}[Proof of Proposition~\ref{pro:upper-vc}]
	Let $X=\Sp^d$ (equipped with its Borel $\sigma$-algebra), $\cR=\cC$ and $\mu=\sigma$.
Since by Lemma~\ref{vcdimcap}  $\vdim(\mathcal{C}_d)=d+2$, Theorem~\ref{thm:disp-vc} yields that for every
$m\ge d+2$ there exists an $m$-point set $P_m\subset \Sp^d$ with
$$\disp_{\cC}(P_m)\le \varepsilon_m:=\frac{3}{\ln 2}\, \frac{d+2}{m}\, \ln\Big(\frac{2e m}{d+2}\Big).$$
In particular, this gives $\disp_{\cC}(m,d)\le \varepsilon_m$ and completes the proof.
\end{proof}

\begin{proof}[Proof of Theorem~\ref{newthm} and inequality~(\ref{eq:prochno-rudolf})]
 We repeat the proof above (and essentially follow \cite{PR24}).
 Using that $\vdim(\mathcal{C}_d)=d+2$ and \cite[Lemma~3.2.3]{BEH+89}, we observe that
 $\vdim(\cL_k)\leq 2 (d+2) k \log_2 (3k)$, in particular,
 $\vdim(\cL_2)=\vdim(\cL)\leq 11(d+2)$.
  Therefore  Theorem~\ref{thm:disp-vc} implies
 \[
 \disp_{\cL_k}(m,d) \le \frac{6}{\ln 2}\, \frac{(d+2) k \log_2 (3k)}{m}\, \ln\Big(\frac{e m}{ (d+2) k \log_2 (3k)}\Big).
\]
for every $m\ge 2 (d+2) k \log_2 (3k)$ (in the case of $\cL$, that is, in the case $k=2$, it is enough to ask $m\ge 11(d+2)$).
\end{proof}

\begin{remark}
Note that	Theorem~\ref{thm:disp-vc} yields the existence of an absolute constant $C>1$ such that for every $\varepsilon\in (0,\frac{1}{2})$
there exists a point set $P_n\subset \cX$ of cardinality $$n\le \frac{C\vdim(\cR)}{\varepsilon}\ln\frac{1}{\varepsilon}$$ with
$\disp(P_n,\cR,\mu)<\varepsilon$. It was shown in \cite[Theorem 2.1]{KPW92}  that in general this bound on $n$ is sharp up to
an absolute  constant  and in \cite[Theorem 3.1]{KPW92} that for sufficiently small $\varepsilon>0$ (depending on $\vdim(\cR)$)
one may take $C$ arbitrarily close to 1. Note however that for the proof of Proposition~\ref{pro:upper-vc} we require bounds
valid for all $\varepsilon\in (0,\frac{1}{2})$.
\end{remark}

\begin{remark}
The notion of VC-dimension can also be used to bound the size of empirical processes or (spherical cap) discrepancy, see \cite{HNWW01,Tal94}.
\end{remark}

The remainder of this section is devoted to the proof of Theorem~\ref{thm:disp-vc}, which is essentially taken from \cite[Appendix~A]{BEH+89} and \cite{HW87}.

A set $P\subset \cX$ is an \textit{$\varepsilon$-traversal} of $\cR$ if $R\cap P\neq \emptyset$ for all
$$
 R\in \cR_{\varepsilon}:=\{R\in \cR\colon \mu(R)> \varepsilon\}.
$$
This definition is more general than the one of (combinatorial) $\varepsilon$-nets introduced in \cite{HW87}.
Theorem~\ref{thm:disp-vc} provides an upper bound on the minimal cardinality of $\varepsilon$-traversals and
 thus  generalizes  the corresponding result for $\varepsilon$-nets in \cite{HW87}.

Given $\varepsilon>0$ and $m\ge 1$ consider the set
\[
Q^m_{\varepsilon}
=\{(x_1,\dots,x_m)\in \cX^m\colon \exists R\in \cR_{\varepsilon}\,\, \mbox{ with } \,\,  R\cap\{x_1,\dots,x_m\}=\emptyset\}
\]
of realizations of $\{X_1,\dots,X_m\}$ not forming an $\varepsilon$-traversal. Then the probability that $\{X_1,\dots,X_m\}$ has dispersion $>\varepsilon$ is given by $\mu^{\otimes m}(Q^m_{\varepsilon})$, where $\mu^{\otimes m}$ denotes the product measure on $X^m$. Also define the set
\begin{align*}
	J^{2m}_{\varepsilon}=\Big\{&(\bar{x},\bar{y})=(x_1,\dots,x_m,y_1,\dots,y_m)\in \cX^{2m}\colon \\
&\exists R\in \cR_{\varepsilon}\, \mbox{ such that } \,  R\cap \{x_1,\dots,x_m\}=\emptyset
\, \mbox{ and } \, |R\cap \{y_1,\dots,y_m\}|\ge\frac{\varepsilon m}{2}\Big\}.
\end{align*}

\begin{lem}\label{lem:Q-bound}
Let  $\varepsilon>0$ and $m\ge 1$ be such that $m\eps \geq 2$. Then
\[
\mu^{\otimes 2m}(J^{2m}_{\varepsilon})
\le \mu^{\otimes m}(Q^m_{\varepsilon})
\le 2 \mu^{\otimes 2m}(J^{2m}_{\varepsilon}).
\]
\end{lem}
\begin{proof}
	If $\varepsilon\ge 1$, then $\cR_{\varepsilon}=\emptyset$ and the statement is trivial.
Thus assume $\varepsilon\in (0,1)$. The first inequality follows since $J^{2m}_{\varepsilon}$ contains a stronger condition
and by Fubini's theorem.

To prove the second bound, note that by Fubini's theorem,
\begin{align*}
\mu^{\otimes 2m}(J^{2m}_{\varepsilon})
&=\int_{\cX^m}\int_{\cX^m}\mathbf{1}_{J^{2m}_{\varepsilon}}(\bar{x},\bar{y})\dd\mu^{\otimes m}(\bar{y})\dd\mu^{\otimes m}(\bar{x})\\
&\ge \int_{Q^m_{\varepsilon}}\int_{\cX^m}\mathbf{1}_{J^{2m}_{\varepsilon}}(\bar{x},\bar{y})\dd\mu^{\otimes m}(\bar{y})\dd\mu^{\otimes m}(\bar{x}).
\end{align*}
We estimate the inner integral uniformly from below. Fix $\bar{x}\in Q^m_{\varepsilon}$.
Then there exists $R\in \cR_{\varepsilon}$ such that $R\cap \{x_1,\dots,x_m\}=\emptyset$.
We fix  one such $R$ and denote it by $R_{\bar{x}}$.
Let $K^{2m}_{\varepsilon}=K^{2m}_{\varepsilon}(\bar{x})$ be the set of $(\bar{x},\bar{y})\in \cX^{2m}$ with $|\{y_1,\dots,y_m\}\cap R_{\bar{x}}|\ge \varepsilon m/2$. Then $K^{2m}_{\varepsilon}\subset J^{2m}_{\varepsilon}$ and hence
\begin{align*}
	\int_{\cX^m}\mathbf{1}_{J^{2m}_{\varepsilon}}(\bar{x},\bar{y})\dd\mu^{\otimes m}(\bar{y})
	&\ge \int_{\cX^m}\mathbf{1}_{K^{2m}_{\varepsilon}}(\bar{x},\bar{y})\dd\mu^{\otimes m}(\bar{y})\\
	&= \int_{\cX^m}\mathbf{1}_{\{|\{y_1,\dots,y_m\}\cap R_{\bar{x}}|\ge \varepsilon m/2\}}(\bar{y})\dd\mu^{\otimes m}(\bar{y}).
\end{align*}
The latter integral is  identical to the probability that $\sum_{i=1}^{m}\xi_i\ge \varepsilon m/2$, where $\xi_i=\mathbf{1}_{Y_i\in R_{\bar{x}}}$ and $Y_1,\dots,Y_m$
are independently distributed according to $\mu$. Then the random variables $\xi_1,\dots,\xi_m$ are independent Bernoulli random variables distributed with paramater
$p=\mu(R_{\bar{x}})>\varepsilon$. We set $p=\varepsilon$ which only decreases the above probability.
Since any median $M$ of the Binomial distribution with   parameters $m, \eps$ satisfies $M\geq m\eps -1$,
we observe that for $m\geq 2/\eps$,
\[
\IP\Big[\sum_{i=1}^{m}\xi_i\ge \varepsilon m/2\Big] \ge \frac{1}{2}.
\]
Therefore,
\[
\mu^{\otimes 2m}(J^{2m}_{\varepsilon})
\ge \frac{1}{2}\int_{Q^m_{\varepsilon}}\dd\mu^{\otimes m}(\bar{x})
=\frac{1}{2}\mu(Q^m_{\varepsilon}),
\]
which proves the lemma.
\end{proof}

We need another lemma. Recall that $\Pi_{\cR}$ was defined in (\ref{shatter}).

\begin{lem}\label{lem:J-bound}
	For all $\varepsilon>0$ and $m\ge 1$,
$$\mu^{\otimes 2m}(J^{2m}_{\varepsilon})\le \Pi_{\cR} (2m) \, 2^{-\varepsilon m/2}.$$
\end{lem}
\begin{proof}
A permutation $\pi$ of $\{1,\dots,2m\}$ acts on $\bar{z}=(z_1,\dots,z_{2m})\in \cX^{2m}$ by $\pi(\bar{z})=(z_{\pi(1)},\dots,z_{\pi(2m)})$ and satisfies $\pi\circ \mu^{\otimes 2m}=\mu^{\otimes 2m}$. Averaging over all permutations yields
\begin{align*}
\mu^{\otimes 2m}(J^{2m}_{\varepsilon})
&=\int_{\cX^{2m}}\mathbf{1}_{J^{2m}_{\varepsilon}}(\bar{z})\dd\mu^{\otimes 2m}(\bar{z})\\
&=\int_{\cX^{2m}}\Big(\frac{1}{(2m)!}\sum_{\pi}\mathbf{1}_{J^{2m}_{\varepsilon}}(\pi(\bar{z}))\Big)\dd\mu^{\otimes 2m}(\bar{z}),
\end{align*}
where the sum is taken over all $(2m)!$ such permutations. We provide a uniform upper bound on  the integrand.

Fix $\bar{z}=(z_1,\dots,z_{2m})\in \cX^{2m}$ and let $Z=\{z_1,\dots,z_{2m}\}$.
Recall that $\pi(\bar{z})\in J_{\varepsilon}^{2m}$ if and only if there is $R\in \cR_{\varepsilon}$ such that
$R\cap \{z_{\pi(1)},\dots,z_{\pi(m)}\}=\emptyset$ and $|R\cap \{z_{\pi(m+1)},\dots,z_{\pi(2m)}\}|\ge \frac{\varepsilon m}{2}$.
In this case we write $\pi\leftrightarrow R$. Without loss of generality, we replace $\cR_{\varepsilon}$ by the smaller set
system
$$\cR_{\varepsilon}|_{\bar{z}}=\{R\cap \{z_{1},\dots,z_{2m}\}\colon R\in \cR_{\varepsilon}\}.$$

Given $R\in \cR_{\varepsilon}|_{\bar{z}}$ define the set of permutations $S_R=\{\pi\colon \pi \leftrightarrow R\}$.
Then $\pi(\bar{z})\in J_{\varepsilon}^{2m}$ if and only if $\pi\in S_R$ for some (possibly non-unique) $R\in \cR_{\varepsilon}|_{\bar{z}}$. This gives
\[
\frac{1}{(2m)!}\sum_{\pi}\mathbf{1}_{J^{2m}_{\varepsilon}}(\pi(\bar{z}))
\le \frac{1}{(2m)!}\sum_{\pi}\sum_{R\in  \cR_{\varepsilon}|_{\bar{z}}}\mathbf{1}_{\pi\in S_R}
=\sum_{R\in  \cR_{\varepsilon}|_{\bar{z}}}\frac{1}{(2m)!}\sum_{\pi} \mathbf{1}_{\pi\in S_R}.
\]

Let $R\in  \cR_{\varepsilon}|_{\bar{z}}$. If $S_R=\emptyset$, omit $R$ from the first sum. If $S_R\neq \emptyset$, then there exists a subset of distinct indices
$I_R=\{i_1,\dots,i_{\ell}\}$ such that $\frac{\varepsilon m}{2}\le \ell\le m$ and $\{z_1,\dots,z_{2m}\}\cap R=\{z_{i_1},\dots,z_{i_{\ell}}\}$ and $z_i\not\in R$ for $i\not\in I_R$.  Note that for a permutation $\pi$, $\pi\in S_R$ if and only if $\pi^{-1}(i_1),\dots,\pi^{-1}(i_{\ell})\in \{m+1,\dots,2m\}$. The fraction of permutations satisfying this can be computed by noting that there are $\binom{m}{\ell}$ ways to map $I_R$ into $\{m+1,\dots,2m\}$ but $\binom{2m}{\ell}$ ways to map $I_R$ into $\{1,\dots,2m\}$. Therefore,
\[
\frac{1}{(2m)!}\sum_{\pi} \mathbf{1}_{\pi\in S_R}
=\frac{\binom{m}{\ell}}{\binom{2m}{\ell}}
=\frac{m(m-1)\cdots (m-\ell+1)}{2m(2m-1)\cdots (2m-\ell+1)}
\le 2^{-\ell}
\le 2^{-\varepsilon m/2}.
\]
We complete the proof by noting that
\[
|R_{\varepsilon}|_{\bar{z}}|\le|\{R\cap \{z_1,\dots,z_{2m}\}\colon R\in \cR\}|\le \Pi_{\cR}(2m).
\]
\end{proof}

These two lemmas yield that for every $\varepsilon>0$ and $m\ge 1$ with $m\varepsilon\ge 2$,
\[
\mu^{\otimes m}(Q^m_{\varepsilon})
\le 2\mu^{\otimes 2m}(J^{2m}_{\varepsilon})
\le 2\Pi_{\cR}(2m)2^{-\varepsilon m/2}.
\]

It remains to bound the shatter function in terms of the VC-dimension.
It is done by the  Sauer-Shelah lemma, which goes back to independent works by
 Vapnik and Chervonenkis, by  Sauer, and by  Shelah, see the references in \cite{Mat99} related to
 Lemma~5.9 there.

\begin{lem}\label{lem:shatter-bound}
Let $\cR\subset \Sigma$ be a set system with $\vdim(\cR)\le d$. Then for $m\geq 1$,
\[
\Pi_{\cR}(m)\le \Phi_d(m):=\sum_{k=0}^{d}\binom{m}{k}.
\]
\end{lem}

\begin{proof}[Proof of Theorem~\ref{thm:disp-vc}]
Let $m\ge d$ and $\varepsilon>0$. If $m\varepsilon<2$, then the bound on the probability is trivially satisfied. Thus, let $m\varepsilon\ge 2$. We bound the sum in Lemma~\ref{lem:shatter-bound} using the standard bound
\[
 \sum_{k=0}^{d}\binom{m}{k}\leq \Big(\frac{me}{d}\Big)^{d},
\]
which follows from
\[
e^d
\ge \Big(1+\frac{d}{m}\Big)^m
\ge \sum_{k=0}^{d}\Big(\frac{d}{m}\Big)^{k}\binom{m}{k}
\ge \Big(\frac{d}{m}\Big)^{d}\sum_{k=0}^{d}\binom{m}{k}.
\]
Combined with Lemmas~\ref{lem:Q-bound}, \ref{lem:J-bound} and \ref{lem:shatter-bound},
we obtain
\[
\mu^{\otimes m}(Q^m_{\varepsilon})
\le 2\Big(\frac{2 e m}{d}\Big)^{d}2^{-\varepsilon m/2}.
\]
This completes the proof.
\end{proof}

%

\subsection*{Acknowledgement}

This research was funded in whole or in part by the Austrian Science Fund (FWF)
[Grant DOI: 10.55776/P32405; 10.55776/J4777]. For open access purposes, the authors have applied a CC BY
public copyright license to any author-accepted manuscript version arising from this submission.

\bibliographystyle{abbrv}
\bibliography{dispersion}

\end{document}